\documentclass[oneside]{amsart}
\usepackage[utf8]{inputenc}
\usepackage{amstext}
\usepackage{amsthm}
\usepackage{amssymb}
\usepackage{graphicx}
\usepackage{esint}

\makeatletter
\numberwithin{equation}{section}
\numberwithin{figure}{section}
\theoremstyle{plain}
\newtheorem{thm}{\protect\theoremname}
\theoremstyle{definition}
\newtheorem*{example*}{\protect\examplename}
\theoremstyle{plain}
\newtheorem{lem}[thm]{\protect\lemmaname}

\makeatother

\providecommand{\examplename}{Example}
\providecommand{\lemmaname}{Lemma}
\providecommand{\theoremname}{Theorem}

\begin{document}
\title[polynomials with a four-term contiguous relation]{Zeros of a table of polynomials satisfying a four-term contiguous
relation}
\author{Jack Luong \and Khang Tran}
\email{khangt@mail.fresnostate.edu \and jackaham\_luongcoln@mail.fresnostate.edu}
\address{Department of Mathematics, California State University, Fresno.\\
5245 North Backer Avenue M/S PB108 Fresno, CA 93740}
\begin{abstract}
For any $A(z),B(z),C(z)\in\mathbb{C}[z]$, we study the zero distribution
of a table of polynomials $\left\{ P_{m,n}(z)\right\} _{m,n\in\mathbb{N}_{0}}$
satisfying the recurrence relation 
\[
P_{m,n}(z)=A(z)P_{m-1,n}(z)+B(z)P_{m,n-1}(z)+C(z)P_{m-1,n-1}(z)
\]
with the initial condition $P_{0,0}(z)=1$ and $P_{-m,-n}(z)=0$ for
all $m,n\in\mathbb{N}$. We show that the zeros of $P_{m,n}(z)$ lie
on a curve whose equation is given explicitly in terms of $A(z),B(z)$,
and $C(z)$. We also study the zero distribution of a case with a
general initial condition. 
\end{abstract}

\maketitle

\section{Introduction}

The study of the zero distribution of a sequence of polynomials $\left\{ S_{N}(z)\right\} _{N\in\mathbb{N}_{0}}$
satisfying a finite recurrence of order $m$ 
\[
\sum_{k=0}^{m}A_{k}(z)S_{N-k}(z)=0,\qquad A_{0}(z)=1,
\]
is of interest to many mathematicians. This sequence includes some
classical sequences of orthogonal polynomials such as the sequence
of Chebyshev polynomials. Orthogonality, in turn, provides information
regarding the locations of zeros of polynomials in the sequence (i.e.,
orthogonality implies reality of zeros). Another approach to find
these locations is to use asymptotic analysis and properties of exponential
polynomials to obtain an optimal curve containing zeros of all these
polynomials (see \cite{ft1,ft2}). In this setting, optimal curve
means that the union of all these zeros forms a dense subset of this
curve. However, even for the case $m=3$ (four-term recurrence), an
explicit equation in terms of the coefficient polynomials, $A_{k}(z)$,
for such an optimal curve is still unknown. Some special cases of
this four-term recurrence have been studied in \cite{adams,err,ghr,tz,tz-1}.

Another approach to the zero distribution of the sequence $\left\{ S_{N}(z)\right\} _{N\in\mathbb{N}_{0}}$
is to produce this sequence from a table of polynomials $\left\{ P_{m,n}(z)\right\} _{m,n\in\mathbb{N}_{0}}$.
This approach necessitates the study of zero distribution of a table
of polynomials satisfying a finite recurrence. Motivated by this approach
and the unsolved four-term recurrence mentioned above, for $A(z),B(z),C(z)\in\mathbb{C}[z]$,
we study the distribution of zeros of a table of polynomials $\left\{ P_{m,n}(z)\right\} _{m,n\in\mathbb{N}_{0}}$
satisfying the recurrence relation 
\[
P_{m,n}(z)+A(z)P_{m-1,n}(z)+B(z)P_{m,n-1}(z)+C(z)P_{m-1,n-1}(z)=0
\]
with the standard initial conditions $P_{0,0}(z)=1$ and $P_{-m,-n}(z)=0$
$\forall m,n\in\mathbb{N}$. Equivalently, this table is generated
by 
\begin{equation}
\sum_{m=0}^{\infty}\sum_{n=0}^{\infty}P_{m,n}(z)s^{m}t^{n}=\frac{1}{1+A(z)s+B(z)t+C(z)st}.\label{eq:genfunc}
\end{equation}
In the case $A(z)=z$, $B(z)=z$, and $C(z)=z$, the numbers $P_{m,n}(-1)$
are the Delannoy numbers which count the number of nearest-neighbor
paths from the origin to $(m,n)$ moving only north, east and northeast
(see \cite{stanley}). When $A(z)=-1$, $B(z)=-1$, and $C(z)=z$,
the numbers $P_{m,n}(0)$ are the binomial coefficients ${\displaystyle \binom{m+n}{m,n}=\frac{(m+n)!}{m!n!}}$.

In terms of its connection with the four-term recurrence sequence
$\left\{ S_{N}(z)\right\} _{N\in\mathbb{N}_{0}}$, letting $t=s^{2}$
in \eqref{eq:genfunc} produces the sequence $S_{N}(z):=\sum_{m+2n=N}P_{m,n}(z)$
which satisfies the general four-term recurrence 
\[
S_{N}(z)+A(z)S_{N-1}(z)+B(z)S_{N-2}(z)+C(z)S_{N-3}(z)=0
\]
with the initial condition $S_{0}(z)=1$ and $S_{-N}(z)=0$ for $N\in\mathbb{N}$.
On the other hand, if we let $t=s$ in \eqref{eq:genfunc}, then with
$D(z):=A(z)+B(z)$, we obtain the sequence $R_{N}(z):=\sum_{m+n=N}P_{m,n}(z)$
satisfying the general three-term recurrence 
\[
R_{N}(z)+D(z)R_{N-1}(z)+C(z)R_{N-2}(z)=0
\]
with the same initial condition. The zeros of these polynomials $R_{N}(z)$
which are not zeros of $C(z)$ lie on the curve defined by \cite{tran}
\[
\Im\frac{D^{2}(z)}{C(z)}=0\qquad\text{and}\qquad0\le\Re\frac{D^{2}(z)}{C(z)}\le4
\]
and are dense there as $N\rightarrow\infty$.

Before stating our theorem, we quickly mention that in the special
case $A(z)=1$, $B(z)=1$, and $C(z)=z$, the diagonal of our table
$\left\{ P_{m,n}(z)\right\} _{m,n\in\mathbb{N}_{0}}$ defined in \eqref{eq:genfunc}
relates to the famous sequence of Legendre polynomials $\left\{ L_{m}(z)\right\} _{m=0}^{\infty}$
(c.f. Lemma \ref{lem:Legendre}) by 
\[
P_{m,m}(z)=z^{m}L_{m}\left(\frac{2}{z}-1\right).
\]
With $\mathcal{Z}(P_{m,n})$ denoting the set of zeros of $P_{m,n}(z)$,
we state our theorem. 
\begin{thm}
\label{thm:mainthm}For any polynomials $A(z),B(z),C(z)\in\mathbb{C}[z]$,
let $\left\{ P_{m,n}(z)\right\} _{m,n=0}^{\infty}$ be the table of
polynomials generated by \eqref{eq:genfunc}. Then for any $m,n\in\mathbb{N}_{0}$,
all the zeros of $P_{m,n}(z)$ which satisfy $A(z)B(z)\ne0$ lie on
the curve $\mathcal{C}$ defined by 
\[
\Im\left(\frac{C(z)}{A(z)B(z)}\right)=0\qquad\text{and}\qquad\Re\left(\frac{C(z)}{A(z)B(z)}\right)\ge1,
\]
and ${\displaystyle \bigcup_{m,n=0}^{\infty}\mathcal{Z}\left(P_{m,n}\right)}$
is dense on $\mathcal{C}$. 
\end{thm}

We will prove this theorem in Section 2 and study a general initial
condition (c.f. Theorem \ref{thm:generaltheorem}) in Section 3. We
end the introduction with an example of Theorem \ref{thm:mainthm}. 
\begin{example*}
In the case $A(z)=1$, $B(z)=z^{2}-2z+2$, and $C(z)=z$, we let $z=x+iy$
and compute 
\begin{align*}
\Im\left(\frac{C(z)}{A(z)B(z)}\right) & =-\frac{y\left(x^{2}+y^{2}-2\right)}{\left((x-1)^{2}+(y-1)^{2}\right)\left((x-1)^{2}+(y+1)^{2}\right)},\\
\Re\left(\frac{C(z)}{A(z)B(z)}\right) & =\frac{x^{3}-2x^{2}+xy^{2}+2x-2y^{2}}{\left((x-1)^{2}+(y-1)^{2}\right)\left((x-1)^{2}+(y+1)^{2}\right)}.
\end{align*}
The equation $\Im\left(C(z)/(A(z)B(z))\right)=0$ implies that $\mathcal{C}$
is a portion of the line $y=0$ and a portion of the circle $x^{2}+y^{2}-2=0$.
In the case $y=0$, the inequality $\Re\left(C(z)/(A(z)B(z))\right)-1\ge0$
yields $(x-1)(2-x)\ge0$. Similarly in the case $x^{2}+y^{2}=2$,
this inequality gives 
\[
\Re\left(\frac{C(z)}{A(z)B(z)}\right)-1=\frac{3-2x}{2(x-1)}\ge0
\]
or equivalently $x\ge1$ since $3-2x>0$ (for $x^{2}+y^{2}=2$). We
conclude that $\mathcal{C}$ is the union of the interval $[1,2]$
and the portion of the circle $x^{2}+y^{2}=2$ with $x\ge1$ (see
Figure \ref{fig:example}).

\begin{figure}
\begin{centering}
\includegraphics[scale=0.5]{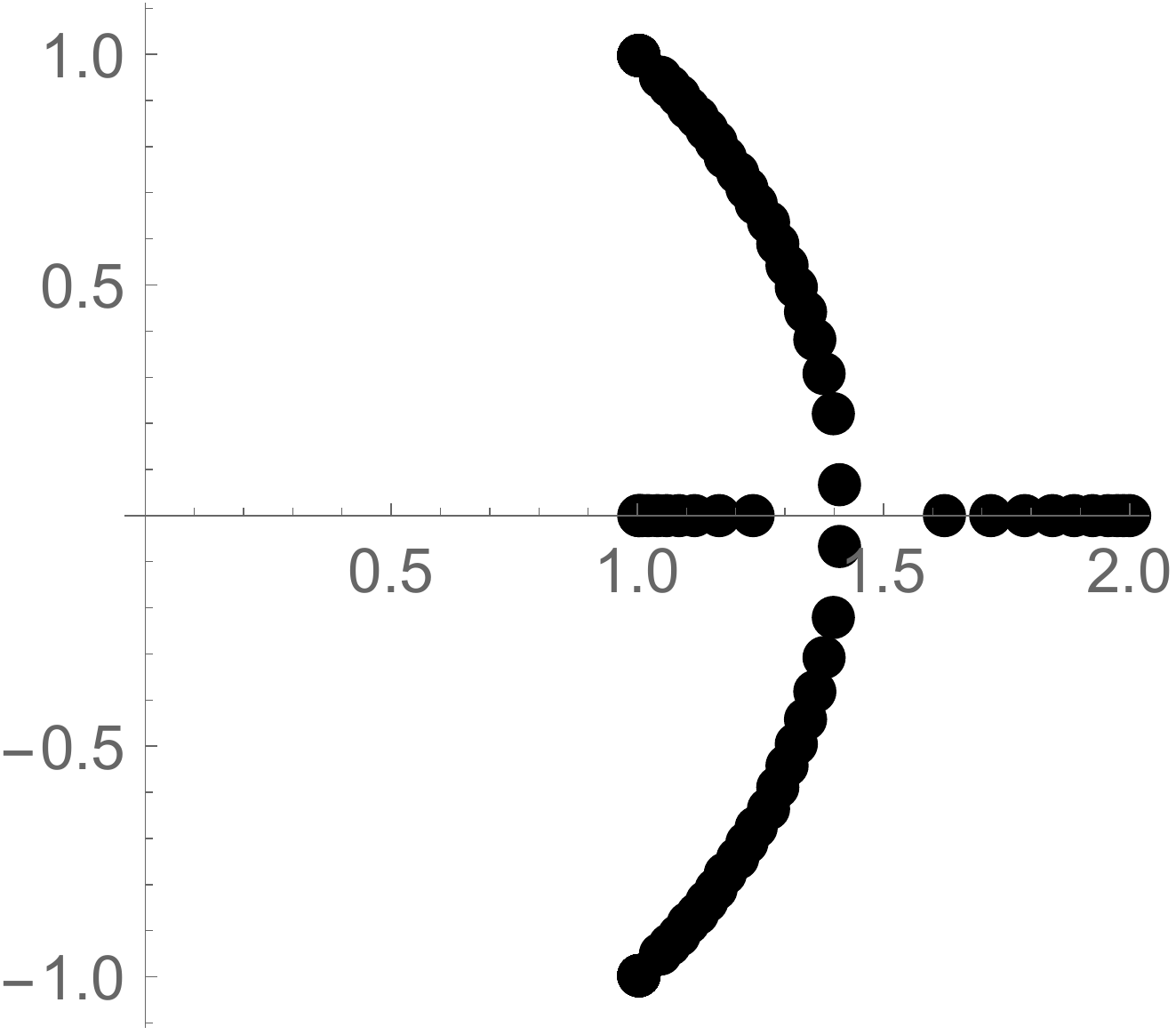} 
\par\end{centering}
\caption{\label{fig:example}Zeros of $P_{50,30}(z)$ when $A(z)=1$, $B(z)=z^{2}-2z+2$,
and $C(z)=z$}
\end{figure}
\end{example*}

\section{Proof of the theorem}

For any $z$ such that $A(z)B(z)\ne0$, we make the substitutions
$s\rightarrow s/A(z)$, $t\rightarrow t/B(z)$ and rewrite \eqref{eq:genfunc}
as 
\[
\sum_{m=0}^{\infty}\sum_{n=0}^{\infty}P_{m,n}(z)\frac{s^{m}}{A^{m}(z)}\frac{t^{n}}{B^{n}(z)}=\frac{1}{1+s+t+C(z)st/A(z)B(z)}.
\]
If we let $\left\{ H_{m,n}(z)\right\} _{m,n=0}^{\infty}$ be the table
of polynomials generated by 
\begin{equation}
\sum_{m=0}^{\infty}\sum_{n=0}^{\infty}H_{m,n}(z)s^{m}t^{n}=\frac{1}{1+s+t+zst},\label{eq:genfuncH}
\end{equation}
then 
\[
A^{m}(z)B^{n}(z)H_{m,n}\left(\frac{C(z)}{A(z)B(z)}\right)=P_{m,n}(z).
\]
Thus to prove \ref{thm:mainthm}, it suffices to prove that for any
$m,n\in\mathbb{N}_{0}$, the zeros of $H_{m,n}(z)$ lie on $[1,\infty)$
and $\bigcup_{m,n=0}^{\infty}\mathcal{Z}\left(H_{m,n}\right)$ is
dense on this interval.

From \eqref{eq:genfuncH}, for small $s,t$, we have

\begin{align*}
\sum_{m=0}^{\infty}\sum_{n=0}^{\infty}H_{m,n}(z)s^{m}t^{n} & =\frac{1}{1+s+t(1+zs)}\\
 & =\frac{1}{(1+s)\left(1+\frac{1+zs}{1+s}t\right)}\\
 & =\frac{1}{1+s}\sum_{n=0}^{\infty}t^{n}(-1)^{n}\left(\frac{1+zs}{1+s}\right)^{n}\\
 & =\sum_{n=0}^{\infty}t^{n}(-1)^{n}\frac{(1+zs)^{n}}{(1+s)^{n+1}}
\end{align*}
and consequently for each $n\in\mathbb{N}_{0}$ 
\begin{equation}
\sum_{m=0}^{\infty}H_{m,n}(z)s^{m}=(-1)^{n}\frac{(1+zs)^{n}}{(1+s)^{n+1}}.\label{eq:genfuncfixedn}
\end{equation}

\begin{lem}
\label{lem:Qmn}For each fixed $n\in\mathbb{N}_{0}$, let $\left\{ Q_{m,n}(z)\right\} _{m\in\mathbb{N}_{0}}$
be the sequence of polynomials generated by 
\begin{equation}
\sum_{m=0}^{\infty}Q_{m,n}(z)s^{m}=\frac{(1+s)^{n}}{1+zs}.\label{eq:Qmngen}
\end{equation}
Then 
\[
Q_{m,n}(z)=\begin{cases}
\sum_{k=0}^{m}(-1)^{m+k}\binom{n}{k}z^{m-k} & \text{ if }m<n\\
(-1)^{m}(z-1)^{n}z^{m-n} & \text{ if }m\ge n.
\end{cases}
\]
\end{lem}

\begin{proof}
From \eqref{eq:Qmngen}, we have 
\[
(1+zs)\sum_{m=0}^{\infty}Q_{m,n}(z)s^{m}=(1+s)^{n}.
\]
We collect the coefficient of $s^{m}$ of both sides and conclude
that 
\[
zQ_{m-1,n}(z)+Q_{m,n}(z)=\begin{cases}
\binom{n}{m} & \text{ if }1\le m\le n\\
0 & \text{ if }m>n
\end{cases}
\]
and the lemma follows from induction by $m$. 
\end{proof}
From Lemma \ref{lem:Qmn}, if $m\ge n$, then the zeros of all the
derivatives (of any orders) of $Q_{m,n}(z)$ lie on the interval $[0,1]$
by the Gauss-Lucas theorem. Thus the reciprocals of the zeros of all
these derivatives lie on $[1,\infty)$. With this observation, the
fact that the zeros of $H_{m,n}(z)$ lie on $[1,\infty)$ follows
from the lemma below. 
\begin{lem}
\label{lem:HQ}For any $m,n\in\mathbb{N}_{0}$, 
\[
H_{m,n}(z)=\frac{z^{m}}{n!}Q_{m+n,n}^{(n)}\left(\frac{1}{z}\right),
\]
where $Q_{m+n,n}^{(n)}(1/z)$ is the $n$-th derivative of $Q_{m+n,n}(z)$
evaluated at $1/z$ . 
\end{lem}

\begin{proof}
We compute the $n$-th derivative of both sides of \eqref{eq:Qmngen}
in $z$ and obtain 
\[
\sum_{m=0}^{\infty}Q_{m,n}^{(n)}(z)s^{m}=(-1)^{n}\frac{n!(1+s)^{n}s^{n}}{(1+zs)^{n+1}}.
\]
By Lemma \ref{lem:Qmn}, for $m<n$, the degree of $Q_{m,n}(z)$ is
$m$ and thus $Q_{m,n}^{(n)}(z)=0$. With the substitutions $z\rightarrow1/z$
and $s\rightarrow sz$, this identity becomes 
\[
\sum_{m=n}^{\infty}Q_{m,n}^{(n)}\left(\frac{1}{z}\right)z^{m}s^{m}=(-1)^{n}\frac{n!(1+sz)^{n}s^{n}z^{n}}{(1+s)^{n+1}}
\]
or equivalently after dividing $s^{n}z^{n}$ on both sides and substituting
$m-n$ by $m$ 
\[
\sum_{m=0}^{\infty}Q_{m+n,n}^{(n)}\left(\frac{1}{z}\right)z^{m}s^{m}=(-1)^{n}\frac{n!(1+sz)^{n}}{(1+s)^{n+1}}.
\]
and the lemma follows from \eqref{eq:genfuncfixedn}. 
\end{proof}
To prove ${\displaystyle \bigcup_{m,n=0}^{\infty}\mathcal{Z}(H_{m,n}(z))}$
is a dense subset of $[1,\infty)$, we will show this assertion holds
for ${\displaystyle \bigcup_{m=0}^{\infty}\mathcal{Z}(H_{m,m}(z))}$.
The lemma below shows the connection between diagonal sequence $\left\{ H_{m,m}(z)\right\} _{m=0}^{\infty}$
and the sequence of Legendre polynomials which is generated by 
\[
\sum_{m=0}^{\infty}L_{m}(z)x^{m}=\frac{1}{\sqrt{1-2zx+x^{2}}}.
\]

\begin{lem}
\label{lem:Legendre}For any $m\in\mathbb{N}_{0}$, 
\[
H_{m,m}(z)=z^{m}L_{m}\left(\frac{2}{z}-1\right).
\]
\end{lem}

\begin{proof}
We first note that for each $z$, there is a small $\delta=\delta(z)>0$
such that \eqref{eq:genfuncH} holds for $|s|,|t|<\delta$. With the
substitution $x=st$ in \eqref{eq:genfuncH}, we deduce from 
\begin{align*}
\sum_{m=0}^{\infty}\sum_{n=0}^{\infty}H_{m,n}(z)x^{m}t^{n-m} & =\frac{1}{1+\frac{x}{t}+t+zx}
\end{align*}
that for each $|x|<\delta^{2}$, $\sum_{m=0}^{\infty}H_{m,m}(z)x^{m}$
is the free coefficient of the Laurent series in $t$ on the annulus
$x/\delta<|t|<\delta$. This coefficient is the residue of 
\[
\frac{1}{t(1+x/t+t+zx)}
\]
at $t=0$. To compute this residue, we evaluate 
\begin{equation}
\frac{1}{2\pi i}\ointctrclockwise_{|t|=\epsilon}\frac{1}{x+(1+zx)t+t^{2}}dt\label{eq:int-xt}
\end{equation}
where $x/\delta<\epsilon<\delta$. With the principal cut, the integrand
has two simple poles at 
\[
t_{1}=\frac{-zx-1+\sqrt{(1+zx)^{2}-4x}}{2}
\]
and 
\[
t_{2}=\frac{-zx-1-\sqrt{(1+zx)^{2}-4x}}{2}.
\]
For small $x$, we have $t_{1}=-x+\mathcal{O}(x^{2})$ and $t_{2}=-1+\mathcal{O}(x)$.
Thus the inequalities $x/\delta<\epsilon<\delta$ imply that, for
small $\delta$, only the pole $t_{1}$ lies in the circle radius
$\epsilon$. The integral \eqref{eq:int-xt} becomes 
\begin{align*}
\frac{1}{t_{1}-t_{2}} & =\frac{1}{\sqrt{(1+zx)^{2}-4x}}\\
 & =\frac{1}{\sqrt{1+(2z-4)x+z^{2}x^{2}}}.
\end{align*}
With this substitution $x\rightarrow x/z$ we conclude 
\[
\sum_{m=0}^{\infty}H_{m,m}(z)\frac{x^{m}}{z^{m}}=\frac{1}{\sqrt{1-2(2/z-1)x+x^{2}}}
\]
and the follows from comparing the right side with the generating
function of Legendre polynomials. 
\end{proof}
Since the union of all the zeros of $L_{m}(z)$, for all $m\in\mathbb{N}_{0}$,
is dense on $[-1,1]$, we conclude from the lemma above that the union
of all the zeros of $H_{m,m}(z)$ is dense on $[1,\infty)$. Moreover,
Lemma \ref{lem:Legendre} also gives the limiting distribution of
the zeros of $H_{m,m}(z)$ on this interval. Indeed, it is known that
(see \cite{assche}) the normalized zero counting measure, 
\[
\frac{1}{m}\sum_{k=1}^{m}\delta_{z_{k,m}}
\]
where $\delta_{z_{k,m}}$ is the Dirac point mass at the zero $z_{k,m}$
of $L_{m}(z)$, has the (weak{*}) asymptotic zero distribution $\omega$
with density 
\[
\frac{d\omega}{dx}=\begin{cases}
\frac{1}{\pi\sqrt{1-x^{2}}}, & \text{if }x\in[-1,1],\\
0, & \text{elsewhere}.
\end{cases}
\]
As a consequence of Lemma \ref{lem:Legendre}, the asymptotic zero
distribution $\mu$ for the normalized zero counting measure of $H_{m,m}(z)$
has the density 

\[
\frac{d\mu}{dx}=\begin{cases}
\frac{2}{\pi x^{2}\sqrt{1-(2/x-1)^{2}}}, & \text{if }x\in[1,\infty),\\
0, & \text{elsewhere.}
\end{cases}
\]

\section{A generalization}

In the previous section, the generating function \eqref{eq:genfuncH}
plays a key role in finding the zero distribution of the polynomials
$P_{m,n}(z)$ generated by \eqref{eq:genfunc}. In this section we
study the zeros of polynomials $\left\{ R_{m,n}(z)\right\} _{m,n=0}^{\infty}$
generated by
\begin{equation}
\sum_{m=0}^{\infty}\sum_{n=0}^{\infty}R_{m,n}(z)s^{m}t^{n}=\frac{N(s,t,z)}{1+s+t+zst}\label{eq:Rgenfuc}
\end{equation}
for any polynomial 
\[
N(s,t,z)=\sum_{i=0}^{I}\sum_{j=0}^{J}\sum_{k=0}^{K}a_{i,j,k}s^{i}t^{j}z^{k}\in\mathbb{R}[s,t,z].
\]
We first collect the $s^{m}t^{n}$-coefficient of both sides of 
\[
\left(\sum_{m=0}^{\infty}\sum_{n=0}^{\infty}R_{m,n}(z)s^{m}t^{n}\right)(1+s+t+zst)=\sum_{i=0}^{I}\sum_{j=0}^{J}\sum_{k=0}^{K}a_{i,j,k}s^{i}t^{j}z^{k}
\]
to conclude 
\begin{equation}
R_{m,n}(z)+R_{m-1,n}(z)+R_{m,n-1}(z)+zR_{m-1,n-1}(z)=\sum_{k=0}^{K}a_{m,n,k}z^{k}\label{eq:recurrenceRmn}
\end{equation}
with the convention that $a_{m,n,k}=0$ if $m>I$ or $n>J$. Thus
for $m>I$ or $n>J$, the polynomials $R_{m,n}(z)$ satisfy the recurrence
\[
R_{m,n}(z)+R_{m-1,n}(z)+R_{m,n-1}(z)+zR_{m-1,n-1}(z)=0,
\]
which is uniquely determined by the denominator of the generating
function, $1+s+t+zst$. From \eqref{eq:recurrenceRmn}, the initial
polynomials of the recurrence $R_{m,n}(z)$ for $m\le I$ and $n\le J$
are determined by the numerator of the generating function $N(s,t,z)$. 

We state our main theorem for this section. 
\begin{thm}
\label{thm:generaltheorem}For any $m,n\in\mathbb{N}_{0}$, the number
of nonreal zeros of $R_{m,n}(z)$, defined as in \eqref{eq:Rgenfuc},
is at most $I+J+2K$. 
\end{thm}

We deduce from \eqref{eq:recurrenceRmn} and induction that the degree
of $R_{m,n}(z)$ is at most 
\[
\max\left(\min(m,n),I+J+K\right).
\]
Thus it suffices to prove Theorem \ref{thm:generaltheorem} for the
case $\min(m,n)\ge I+J+K$.

From \eqref{eq:Rgenfuc}, we conclude 
\begin{align*}
\sum_{m=0}^{\infty}\sum_{n=0}^{\infty}R_{m,n}(z)s^{m}t^{n} & =\frac{\sum_{i=0}^{I}\sum_{j=0}^{J}\sum_{k=0}^{K}a_{i,j,k}s^{i}t^{j}z^{k}}{1+s+t+zst}\\
 & =\sum_{i=0}^{I}\sum_{j=0}^{J}\sum_{k=0}^{K}a_{i,j,k}s^{i}t^{j}z^{k}\sum_{m=0}^{\infty}\sum_{n=0}^{\infty}H_{m,n}(z)s^{m}t^{n}.
\end{align*}
 By equating the coefficients of $s^{m}t^{n}$, the equation above
yields 
\begin{align*}
R_{m,n}(z)= & \sum_{i=0}^{I}\sum_{j=0}^{J}\sum_{k=0}^{K}a_{m-i,n-j,k}z^{k}H_{m-i,n-j}(z)\\
\stackrel{\text{Lemma }\ref{lem:HQ}}{=} & \sum_{i=0}^{I}\sum_{j=0}^{J}\sum_{k=0}^{K}a_{m-i,n-j,k}\frac{1}{(n-j)!}z^{m-i+k}Q_{m+n-i-j,n-j}^{(n-j)}\left(\frac{1}{z}\right).
\end{align*}
With the substitution $z$ by $1/z$, this equation implies that $z^{m+K}R_{m,n}(1/z)$
is 
\begin{equation}
\sum_{i=0}^{I}\sum_{j=0}^{J}\sum_{k=0}^{K}a_{m-i,n-j,k}\frac{1}{(n-j)!}z^{i+K-k}Q_{m+n-i-j,n-j}^{(n-j)}\left(z\right).\label{eq:recipRmn}
\end{equation}
To find an upper bound for the number of nonreal zeros of $R_{m,n}(z)$,
which is the same as that of $R_{m,n}(1/z)$, we let $D$ be the differential
operator and consider the lemma below. 
\begin{lem}
\label{lem:diffoperator}For any polynomial $f(z)$ and $m,n\in\mathbb{N}_{0}$
\[
z^{m}D^{n}(f)=\sum_{i=0}^{m}{m \choose i}(-1)^{i}D^{n-i}(z^{m-i}f)(n)_{i}
\]
where $(n)_{i}:=n(n-1)(n-2)\cdots(n-i+1)$ and $(n)_{0}=1$. 
\end{lem}

\begin{proof}
We prove by induction on $m$ by assuming the statement holds up to
$m$ and show it works for $m+1$. Indeed, 
\begin{align*}
z^{m+1}D^{n}(f) & =\sum_{i=0}^{m}(-1)^{i}{m \choose i}zD^{n-i}(z^{m-i}f)(n)_{i},
\end{align*}
where by induction hypothesis (for $m=1$), we have 
\[
zD^{n-i}(z^{m-i}f)=D^{n-i}(z^{m+1-i}f)-(n-i)D^{n-i-1}(z^{m-i}f).
\]
Thus $z^{m+1}D^{n}(f)$ is 
\[
\sum_{i=0}^{m}(-1)^{i}\binom{m}{i}(n)_{i}D^{n-i}(z^{m+1-i}f)-\sum_{i=0}^{m}(-1)^{i}\binom{m}{i}(n)_{i+1}D^{n-i-1}(z^{m-i}f).
\]
We replace $i+1$ by $i$ in the second sum and use the convention
that $\binom{m}{k}=0$ if $k<0$ or $k>m$ to rewrite the expression
above as 
\[
\sum_{i=0}^{m+1}(-1)^{i}\binom{m}{i}(n)_{i}D^{n-i}(z^{m+1-i}f)+\sum_{i=0}^{m+1}(-1)^{i}\binom{m}{i-1}(n)_{i}D^{n-i}(z^{m-i+1}f).
\]
The lemma then follows from the binomial identity 
\[
\binom{m}{i}+\binom{m}{i-1}=\binom{m+1}{i}.\qedhere
\]
\end{proof}
We apply Lemma \ref{lem:diffoperator} and rewrite \eqref{eq:recipRmn}
as 
\begin{align*}
 & \sum_{i=0}^{I}\sum_{j=0}^{J}\sum_{k=0}^{K}\sum_{l=0}^{K+i-k}\frac{a_{m-i,n-j,k}}{(n-j)!}\binom{i+K-k}{l}(-1)^{l}\\
\times & D^{n-j-l}(z^{i+K-k-l}Q_{m+n-i-j,n-j}(z))(n-j)_{l}.
\end{align*}
Since $n\ge I+J$, this expression is $D^{n-J-I-K}$ of 
\begin{align}
 & \sum_{i=0}^{I}\sum_{j=0}^{J}\sum_{k=0}^{K}\sum_{l=0}^{K+i-k}\frac{a_{m-i,n-j,k}}{(n-j)!}\binom{i+K-k}{l}(-1)^{l}\nonumber \\
\times & D^{I+J+K-j-l}(z^{i+K-k-l}Q_{m+n-i-j,n-j}(z))(n-j)_{l}.\label{eq:quadsumderivQ}
\end{align}
With $m+n-i-j\ge n-j$, Lemma \ref{lem:Qmn} gives 
\[
z^{i+K-k-l}Q_{m+n-i-j,n-j}(z)=(-1)^{m+n-i-j}(z-1)^{n-j}z^{m+K-k-l}.
\]
The high order derivatives of the right side are given in the lemma
below. 
\begin{lem}
\label{lem:derivQmn}For any $a,b\in\mathbb{N}_{0}$ and $m\le\min(a,b)$
\[
D^{m}((z-1)^{a}z^{b})=(z-1)^{a-m}z^{b-m}P(z),
\]
where $P(z)$ is a polynomial of degree at most $m$. 
\end{lem}

\begin{proof}
We prove by induction on $m$ for the case $a\le b$ and the same
argument will work for $a>b$. The claim holds trivially for $m=0$.
If this claim holds for some $m<a\le b$, then 
\begin{align*}
D^{m+1}((z-1)^{a}z^{b}) & =D((z-1)^{a-m}z^{b-m}P(z)).
\end{align*}
With the product rule, the right side becomes 
\[
(z-1)^{a-1-m}z^{b-1-m}\left((a-m)zP(z)+(b-m)(z-1)P(z)+(z-1)zP'(z)\right)
\]
and the lemma follows. 
\end{proof}
With the note that $\min(m,n)\ge I+J+K$, we apply Lemma \ref{lem:derivQmn}
to rewrite \eqref{eq:quadsumderivQ} as 
\begin{align*}
 & \sum_{i=0}^{I}\sum_{j=0}^{J}\sum_{k=0}^{K}\sum_{l=0}^{K+i-k}\frac{a_{m-i,n-j,k}}{(n-j)!}\binom{i+K-k}{l}(-1)^{m+n-i-j+l}(n-j)_{l}\\
 & \times(z-1)^{n-I-J-K+l}z^{m-k-I-J+j}P_{m,n,i,j,k,l}(z)\\
= & (z-1)^{n-I-J-K}z^{m-I-J-K}\sum_{i=0}^{I}\sum_{j=0}^{J}\sum_{k=0}^{K}\sum_{l=0}^{K+i-k}\\
 & \frac{a_{m-i,n-j,k}}{(n-j)!}\binom{i+K-k}{l}(-1)^{m+n-i-j+l}(n-j)_{l}(z-1)^{l}z^{K-k+j}P_{m,n,i,j,k,l}(z)
\end{align*}
where $P_{m,n,i,j,k,l}(z)$ is a polynomial of degree at most $I+J+K-j-l$.
Since the degree of 
\[
(z-1)^{l}z^{K-k+j}P_{m,n,i,j,k,l}(z)
\]
is at most 
\[
2K-k+I+J\le I+J+2K,
\]
so is the degree of 
\begin{align*}
 & \sum_{i=0}^{I}\sum_{j=0}^{J}\sum_{k=0}^{K}\sum_{l=0}^{K+i-k}\frac{a_{m-i,n-j,k}}{(n-j)!}\binom{i+K-k}{l}(-1)^{m+n-i-j+l}\\
 & \times(n-j)_{l}(z-1)^{l}z^{K-k+j}P_{m,n,i,j,k,l}(z)
\end{align*}
and consequently the quadruple sum above has at most $I+J+2K$ non-real
zeros. Theorem \ref{thm:generaltheorem} follows from the fact that
the differential operator does not increase the number of non-real
zeros of a polynomial.

\textbf{Acknowledgement.} The authors thank the support of Erwin Schrödinger
International Institute for Mathematics and Physics for the workshop
on Optimal Point Configurations on Manifolds. We also thank the referees
for providing corrections and suggestions to improve the paper.

\end{document}